\documentclass[article,12pt]{amsart}
\usepackage{lineno,hyperref,enumitem}
\usepackage{amsmath,amssymb,dsfont}
\newtheorem{thm}{Theorem}[section]
\newtheorem{defn}[thm]{Definition}
\newtheorem{prop}[thm]{Proposition}

\newtheorem{rem}[thm]{Remark}

\newtheorem{ex}[thm]{Example}

\newcommand{\ds}{\displaystyle}

\usepackage{listings}
\usepackage{xcolor}

\begin{document}

\title[Short note on phase retrievable weaving fusion frames]{Short note on phase retrievable weaving fusion frames}

%----------Author 1
%\author[D. Haldar]{Debasis Haldar}
%\address{Department of Mathematics\\ NIT Rourkela\\ Rourkela 769008\\ India}
%\email{debamath.haldar@gmail.com}
%----------Author 2
%\author[D. Singh]{Divya Singh}
%\address{Department of Mathematics\\ NIT Rourkela\\ Rourkela 769008\\ India}
%\email{divya\_allduniv@yahoo.com}

%----------Author 1
\author[A. Bhardwaj]{Avinash Bhardwaj}

\author[A. Bhandari]{Animesh Bhandari$^*$}
\address{Department of Mathematics\\ SRM University, {\it AP} - Andhra Pradesh\\ Neeru Konda, Amaravati, Andhra Pradesh - 522240\\ India}
\email{avinashbhardwaj320@gmail.com,bhandari.animesh@gmail.com, animesh.b@srmap.edu.in}
%\thanks{Third author was supported by DST-SERB project MTR/2017/000797}
\subjclass[2010]{42C15, 46A32, 47A05\\ * is the corresponding author}

%\keywords{$p$-adic setting, Besselet, multiframelet, multiframelet operator}

\begin{abstract}
Fusion frames are extensively studied due to their effectiveness in recovering signals from large-scale data. They are applicable in distributed processing, wireless sensor networks, and packet encoding systems due to their robustness and redundancy. Motivated by the foundational work of Bemrose et al.\cite{Be16} and Balan\cite{Ba13}, this paper investigates the theoretical properties and characterizations of phase retrievable weaving fusion frames. These frames offer enhanced redundancy and stability in signal reconstruction. We present key results that deepen the understanding of their structure and behaviour. Lastly, an application involving probabilistic erasure is explored to demonstrate their practical utility.

%Furthermore, several examples are discussed to validate the results.\\

\end{abstract}

\maketitle

\section{Introduction}
The concept of Hilbert space frames was introduced by Duffin and Schaeffer in 1952 \cite{DuSc52}.
Its significance gained prominence through the influential work of Daubechies, Grossman, and Meyer in 1986 \cite{DaGrMa86}.
Since then, frame theory has found widespread applications in operator theory \cite{HaLa00}, harmonic analysis \cite{Gr01}, wavelet analysis \cite{GuSh18}, and signal/image processing \cite{Fe99, ArZe16}.
It is also instrumental in sensor networks \cite{CaKuLiRo07}, data analysis \cite{CaKu12}, and in the study of retro Banach frames \cite{VeVaSi19}.
The theory has evolved through generalizations like fusion frames \cite{CaKu04}, $G$-frames \cite{Su06}, $K$-frames \cite{Ga12}, and $K$-fusion frames \cite{Bh17}.
These developments have enriched the literature and extended frame theory's utility across mathematics and engineering.

Classical frames and their weaving counterparts offer stable and redundant signal representations, playing a key role in signal processing and functional analysis. In distributed environments, their limitations arise prominently, especially when handling projections onto multiple subspaces. This becomes crucial in applications like wireless sensor networks and packet encoding, where robustness and flexibility are essential. Fusion frames and their weaving forms extend classical frames to weighted subspace collections, overcoming these challenges.
They enhance adaptability and stability in distributed processing, proving effective tools in real-world signal reconstruction.
Additionally, weaving fusion frames have profound implications in theoretical areas such as the Kadison-Singer problem and optimal subspace packing.

Over the past few years, extensive research has been conducted on phase retrieval and norm retrieval frames. However, phase retrievable fusion frames and their weaving forms remain unexplored.
In particular, the application of weaving fusion frames in the context of probabilistic erasure has not been addressed. To bridge these gaps, we investigate various properties and characterizations of phase retrievable weaving fusion frames. This study sheds light on their structural behaviour and theoretical foundations. Furthermore, we explore their potential in enhancing robustness under probabilistic erasure scenarios.

Phase-retrievable weaving fusion frames find natural applications in quantum information theory, particularly in the design of quantum measurement schemes where only intensity (magnitude-squared) information is available. In such scenarios, the quantum system's state is not directly accessible due to phase loss in measurements. These frames help reconstruct a pure or mixed quantum state using intensity data from multiple fusion frame components.

In this context, the Hilbert space $\mathcal H$ represents the state space of the quantum system.
Each fusion frame subspace models a quantum observable or a measurement channel. Weaving fusion frames corresponds to different measurement configurations. The phase retrievability ensures that the entire quantum state can still be reconstructed from intensity-only data across these configurations. The weaving ensures robustness and redundancy, while phase retrievability ensures faithful recovery despite missing phase information  in optical and quantum communication systems.
%To support our findings, several examples are discussed, illustrating the practical relevance and applicability of the results. These examples validate our theoretical results and demonstrate the utility of weaving fusion frames in various contexts.

Throughout this paper, $\mathcal{H}$ is a separable Hilbert space and $\mathcal H^n$ is the $n$-dimensional Hilbert space. We denote by $\mathcal{L}(\mathcal{H}_1, \mathcal{H}_2)$ the space of all bounded linear operators from  $\mathcal{H}_1$ into $\mathcal{H}_2$, and $\mathcal L(\mathcal H)$ for $\mathcal L(\mathcal H, \mathcal H)$. For $T \in \mathcal{L}\mathcal{(H)}$, we denote $D(T), N(T)$ and $R(T)$ for domain, null space and range of $T$, respectively. For a collection of closed subspaces $\mathcal W_i$ of $\mathcal H$ and scalars $w_i$, $i\in I$, the weighted collection of closed subspaces $\lbrace (\mathcal W_i , w_i) \rbrace_{i \in I}$ is denoted by $\mathcal W_w$. We consider $I$ to be countable index set, $\mathcal I$ is the identity operator and $P_{\mathcal V}$ is the orthogonal projection onto $\mathcal V$.

\section{Preliminaries}\label{Sec-Preli}

Before diving into the main sections, throughout this section we recall basic definitions and results needed in this paper. For detailed discussion regarding frames and their applications we refer \cite{CaKu12, Sh20, Ch16, Ha07, Ko23}.

\subsection{Frame} A collection  $\{ f_i \}_{i\in I}$ in $\mathcal{H}$ is called a \emph{frame} if there exist constants $A,B >0$ such that \begin{equation}\label{Eq:Frame} A\|f\|^2~ \leq ~\sum_{i\in I} |\langle f,f_i\rangle|^2 ~\leq ~B\|f\|^2,\end{equation}for all $f \in \mathcal{H}$. The numbers $A, B$ are called \emph{frame bounds}. The supremum over all $A$'s and infimum over all $B$'s satisfying above inequality are called the \emph{optimal frame bounds}.
If a collection satisfies only the right inequality in (\ref{Eq:Frame}), it is called a {\it Bessel sequence}.

Given a frame $\{f_i\}_{i\in I}$ for $\mathcal{H}$, the \emph{pre-frame operator} or \emph{synthesis operator} is a bounded linear operator $T: l^2(I)\rightarrow \mathcal{H}$ and is defined by $T\{c_i\}_{i \in I} = \ds \sum_{i\in I} c_i f_i$. The adjoint of $T$, $T^*: \mathcal{H} \rightarrow l^2(I)$, given by $T^*f = \{\langle f, f_i\rangle\}_{i \in I}$, is called the \emph{analysis operator}. The \emph{frame operator}, $S=TT^*: \mathcal{H}\rightarrow \mathcal{H}$, is defined by $$Sf=TT^*f = \sum_{i\in I} \langle f, f_i\rangle f_i.$$ It is well-known that the frame operator is bounded, positive, self adjoint and invertible.

\subsection{Fusion frame}\label{Fusion} Consider a weighted collection of closed subspaces, $\mathcal W_w=\{(\mathcal W_i, w_i)\}_{i \in I}$, of $\mathcal H$. Then $\mathcal W_w$  is said to be a fusion frame for $\mathcal H$, if there exist constants $0<A \leq B <\infty$ satisfying
\begin{equation}\label{Eq:Fus-frame}
	A\|f\|^2 \leq \sum_{i \in I} w_i ^2 \|P_{\mathcal W_i}f\|^2 \leq B\|f\|^2, \end{equation}
where $P_{\mathcal W_i}$ is the orthogonal projection from $\mathcal H$ onto $\mathcal W_i$. Analogously, the frame operator for a fusion frame can be defined similar to that of a classical frame and its explicit representation is $S_{\mathcal W} (f)=\sum \limits_{i \in I} w_i ^2 P_{\mathcal W_i}(f)$, which is also invertible.
Using this invertibility, every $f \in \mathcal H$ can be reconstructed by its fusion frame measurements $\lbrace w_i P_{\mathcal W_i} f \rbrace_{i \in I}$ as
\begin{equation}\label{Eq:Fusion-frame-recons} f=\sum_{i \in I} v_i S_\mathcal W ^{-1} (v_i P_{\mathcal W_i} f) .\end{equation}

%\subsection{$K$-fusion frame}
%In \cite{Bh17}, authors, introduced a generalization of fusion frame, $K$-fusion frame, and scrutinized the equivalency between atomic subspaces and $K$-fusion frames. $K$-fusion frames are used to reconstruct signals from range of a bounded linear operator $K$.

%\begin{defn}($K$-fusion frame)
%	Let $K\in \mathcal L(\mathcal H)$,  $\mathcal W_v=\lbrace \mathcal (W_i,v_i) \rbrace_{i \in I}$ be a weighted collection of closed subspaces of $\mathcal H$. Then   $\mathcal W_v$ is said to be a $K$-fusion frame for $\mathcal H$  if there exist positive constants $A, B$ such that for all $f \in \mathcal H$ we have
%	\begin{equation}\label{Eq:K-Fus-frame}
%		A\|K^*f\|^2 \leq \sum_{i \in I} v_i ^2 \|P_{\mathcal W_i}f\|^2 \leq B\|f\|^2. \end{equation}
%\end{defn}
\subsection{Weaving Fusion frames}
Let   $\mathcal V_v=\{(\mathcal V_i, v_i)\}_{i \in I}$ and  $\mathcal W_w=\{(\mathcal W_i, w_i)\}_{i \in I}$ be two fusion frames for $\mathcal H$. Then they are called weaving fusion frames for $\mathcal H$ if for every $\sigma \subset I$ and every $f \in \mathcal H$ there exist finite positive constants $A \leq B$ so that $\{(\mathcal V_i, v_i)\}_{i \in \sigma} \cup \{(\mathcal W_i, w_i)\}_{i \in \sigma^c}$ is a fusion frame for $\mathcal H$ with the universal bounds $A \leq B$, i.e.  the following inequality is satisfied:
\begin{equation}\label{weaving}
A \|f\|^2 \leq \sum_{i \in \sigma} v_i ^2 \|P_{\mathcal V_i}f\|^2 +  \sum_{i \in \sigma^c} w_i ^2 \|P_{\mathcal W_i}f\|^2\leq B\|f\|^2.
\end{equation}

\begin{ex}\label{ex1}
Let us consider an orthonormal basis $\{e_n\}_{n=1}^\infty$ for $\mathcal H$. Suppose for every $n$, $\mathcal V_n = \text{span} \{e_n\}$ and $\mathcal W_n = \text{span} \{e_n, e_{n+1}\}$. Since for every $f \in \mathcal H$ and $\sigma \subset \{1, 2, \cdots\}$ we have,
$$\|f\|^2 \leq \sum_{n \in \sigma}  \|P_{\mathcal V_n}f\|^2 +  \sum_{n \in \sigma^c}  \|P_{\mathcal W_n}f\|^2\leq 2\|f\|^2.$$

Therefore, $\{(\mathcal V_n, 1)\}_{n=1}^\infty$ and $\{(\mathcal W_n, 1)\}_{n=1}^\infty$ are weaving fusion frames for $\mathcal H$.
\end{ex}

In the following example we discuss a non-example of weaving fusion frames.

\begin{ex}\label{ex2}
Let us consider an orthonormal basis $\{e_n\}_{n=1}^\infty$ for $\mathcal H$. Suppose for every $n$, $\mathcal V_n = \text{span} \{e_n\}$ and $\mathcal W_1 = \text{span} \{e_2\}$, $\mathcal W_2 = \text{span} \{e_1\}$ and $\mathcal W_n = \text{span} \{e_n\}$ for $n \geq 3$. Since for  $\sigma =\{2\} \subset \{1, 2, \cdots\}$ we have,
$$\sum_{n \in \sigma}  \|P_{\mathcal W_n}e_2\|^2 +  \sum_{n \in \sigma^c}  \|P_{\mathcal V_n}e_2\|^2=0$$

Thus, $\{(\mathcal V_n, 1)\}_{n=1}^\infty$ and $\{(\mathcal W_n, 1)\}_{n=1}^\infty$ are not weaving fusion frames for $\mathcal H$.
\end{ex}

\begin{rem}\label{Remark}
For weaving fusion frames we can define the associated weaving synthesis, analysis and weaving fusion frame operators analogous to the subsection \ref{Fusion}.
\end{rem}

For detailed discussion regarding weaving fusion frames we refer to \cite{De17, Bh25}. 

We present a well-known definition and theorems which serve as a key tool in proving our main results. This theorem provides the necessary foundation to advance our arguments effectively.

\begin{defn}\label{phase-retrieval}\cite{Ju19}
	Let $\{f_i\}_{i=1}^\infty$ be a frame for $\mathcal{H}$.  We say that $\{f_i\}_{i=1}^\infty$ \emph{posses phase retrieval} for $\mathcal{H}$ whenever $ g, j \in \mathcal{H}$ and $|\langle g, f_i \rangle| = |\langle j, f_i \rangle|$ for every $n$, 
	then $g = hj$, where $|h| = 1$.
\end{defn}

\begin{thm}\label{injective-complement} (Comlementary Property)\cite{Ra06}
	Let $(f_i)_{i=1}^m$ be a frame for $\mathcal{H}^n$. Then the following are equivalent:
	\begin{enumerate}
		\item The map $\gamma:  \mathcal H^n/ \{\pm{1}\} \rightarrow \mathbb R^m$ is injective.
		\item  For every $\sigma \subset \{1,2,\ldots,m\}$, either $\{f_i\}_{i \in \sigma}$ spans $\mathcal{H}^n$ or $\{f_i\}_{i \in \sigma^c}$ spans $\mathcal{H}^n$.		
 	\end{enumerate}
 	
  \end{thm}

%\begin{rem}
%It is to be noted that Example \ref{ex1} illustrates a case of phase retrievable weaving fusion frames, whereas Example \ref{ex2} serves as a counterexample to highlight when this property does not hold.
%\end{rem}

%We discuss a fundamental well-known result that serves as a key tool in establishing a necessary and sufficient condition for phase retrievable weaving fusion frames.

\begin{thm}\label{Balan}\cite{Ba13}
Let $\text{Sym}(\mathcal H^n)$ be the collection of symmetric operators on an $n$-dimensional Hilbert space $\mathcal H^n$ and $S^{1,1}(\mathcal H^n)=\{T \in Sym(\mathcal H^n): \dim R(T)\leq2\}$. Then the following are satisfied:
\begin{enumerate}
\item $S^{1,0}(\mathcal H^n)=\{T=f\otimes f: f \in \mathcal H^n\}$.
\item For every $T\in S^{1,1}(\mathcal H^n)$, there exist $T_1, T_2 \in S^{1,0}(\mathcal H^n)$ so that $T=T_1 - T_2$.
\item For every $T\in S^{1,1}(\mathcal H^n)$, there exist $f, g \in \mathcal H^n$ so that $T=\frac{1}{2} (f\otimes g + g\otimes f)$.
\end{enumerate}
\end{thm}

\section{Main Results}\label{Sec-char}

In this section we define phase retrievable weaving fusion frames and discuss their various characterizations. 

\begin{defn}
	Let $f, g \in \mathcal H^n$. Then $f$ and $g$ are said to be similar if there exists $h \in \mathbb T=\{h \in \mathbb C: |h|=1\}$ so that $g=hf$ and it is denoted as $f \sim g$. Suppose $\{(\mathcal V_i, v_i)\}_{i=1}^m$ and $\{(\mathcal W_i, w_i)\}_{i=1}^m$ are weaving fusion frames for $\mathcal H^n$. Define a non-linear operator $\gamma_\sigma$ on the quotient space $\hat{\mathcal H^n}=\mathcal H^n/\sim$ as $\gamma_\sigma: \hat{\mathcal H^n} \rightarrow \mathbb R^m$ and is defined by,
	
		\begin{equation}\label{eq1}
		\left( (\gamma_\sigma(\hat{f}))_i \right)_i^{m} =
		\begin{cases}
			v_i \| P_{V_i} f \| & : i \in \sigma \\
			w_i \| P_{W_i} f \| & : i \in \sigma^c
		\end{cases}
	\end{equation}
	 
	 where $\sigma$ is a subset  of $ \{1, 2, \cdots , m\}$. If for every $\sigma \subset \{1, 2, \cdots , m\}$, $\gamma_\sigma$ is one-one , then $\{(\mathcal V_i, v_i)\}_{i=1}^m$ and $\{(\mathcal W_i, w_i)\}_{i=1}^m$ are said to be phase retrievable weaving fusion frames for $\mathcal H^n$.
	
	 For any two phase retrievable weaving fusion frames $\{(\mathcal V_i, v_i)\}_{i=1}^m$ and $\{(\mathcal W_i, w_i)\}_{i=1}^m$, the associated weaving phase lift operator $\mathcal F_\sigma: Sym(\mathcal H^n) \rightarrow \mathbb R^m$ is defined as, 
	 \begin{equation}
	 	\left( (\mathcal{F}_\sigma T)_i \right)_{i=1}^m =
	 	\begin{cases}
	 		v_i^2 \, \operatorname{tr}(P_{V_i} T) & : i \in \sigma \\
	 		w_i^2 \, \operatorname{tr}(P_{W_i} T) & : i \in \sigma^c
	 	\end{cases}
	 \end{equation}

	 where $\text{Sym}(\mathcal H^n)$ is the collection of symmetric operators on an $n$-dimensional Hilbert space $\mathcal H^n$ and $\sigma \subset \{1, 2, \cdots , m\}.$ \end{defn}
	 
	 We now present an example to demonstrate the preceding definition.
	 
	 \begin{ex}\label{ex3}
	 		
	 	 Let us consider a separable Hilbert space $\mathcal{H} = \mathbb{R}^2$. Suppose
	 	$	 		
	 	\{(\mathcal V_i, 1)\}_{i=1}^3 = \left\{
	 		\left( \operatorname{span} \left( \begin{bmatrix} 1 \\ 0 \end{bmatrix} \right), 1 \right),
	 		\left( \operatorname{span} \left( \begin{bmatrix} 0 \\ 1 \end{bmatrix} \right), 1 \right),
	 		\left( \operatorname{span} \left( \begin{bmatrix} 1 \\ 2 \end{bmatrix} \right), 1 \right)
	 		\right\}
	 $
	 		
	 		and
	 		
	 		$
	 			\{(\mathcal W_i, 1)\}_{i=1}^3 =\left\{
	 			\left( \operatorname{span} \left( \begin{bmatrix} \frac{1}{\sqrt{2}} \\ \frac{1}{\sqrt{2}} \end{bmatrix} \right), 1 \right),
	 			\left( \operatorname{span} \left( \begin{bmatrix} \frac{1}{\sqrt{2}} \\ -\frac{1}{\sqrt{2}} \end{bmatrix} \right), 1 \right),
	 			\left( \operatorname{span} \left( \begin{bmatrix} 2 \\ 1 \end{bmatrix} \right), 1 \right)
	 			\right\}
$	 						
	 		 We claim that the  fusion frames 
	 		%It is easy to see that for every subset $\sigma \subset \{1, 2, 3\},$ equation ~\eqref{eq1} is satisfied. Thus, 
	 		$\{(\mathcal V_i, v_i)\}_{i=1}^3$ and $\{(\mathcal W_i, w_i)\}_{i=1}^3$   are phase retrievable weaving fusion frames for $\mathbb R^2$. To verify the same let us assume  $f = \left( \begin{bmatrix} f_1 \\ f_2 \end{bmatrix} \right)$ with 	 $\gamma_\sigma: \hat{\mathbb R^2} \rightarrow \mathbb R^3$ defined by  
	 		\[
	 		\left( (\gamma_\sigma(\hat{f})_i) \right)_i^{3} =
	 		\begin{cases}
	 			v_i \| P_{V_i} f \| & : i \in \sigma \\
	 			w_i \| P_{W_i} f \| & : i \in \sigma^c,
	 		\end{cases}
	 		\]
	 	where $\hat{f} = \{\pm {f}\}$ and $\sigma \subset \{1, 2, 3\}$. If $\gamma_\sigma$ is one-one for every $\sigma \subset \{1, 2, 3\}$, then $\{(\mathcal V_i, v_i)\}_{i=1}^3$ and $\{(\mathcal W_i, w_i)\}_{i=1}^3$ are said to be phase retrievable weaving fusion frames for $\mathbb R^2$. Here we prove the injectivity of $\gamma_\sigma$.
	 	
	 	 \vspace{0.2cm}
	 	
	 	{\bf Case 1:}  $\sigma$ = $\emptyset$ and let us assume that
	 	$$\left( (\gamma_\sigma(\hat{f})_i) \right)_i^{3} = \left( (\gamma_\sigma(\hat{g})_i) \right)_i^{3}.  $$
		 
	 %	Then we have $\left( 1\|P_{W_1} f\|,\; 1\|P_{W_2} f\|,\; 1\|P_{W_3} f\| \right)
	 %	= 
	 %	\left( 1\|P_{W_1} g\|,\; 1\|P_{W_2} g\|,\; 1\|P_{W_3} g\| \right)$
	 	
	Therefore, $ \left( 
	 \frac{|f_1 + f_2|}{\sqrt{2}}, \;
	 \frac{|f_1 - f_2|}{\sqrt{2}}, \;
	 \frac{|2f_1 + f_2|}{\sqrt{5}}
	 \right)
	 =
	 \left( 
	 \frac{|g_1 + g_2|}{\sqrt{2}}, \;
	 \frac{|g_1 - g_2|}{\sqrt{2}}, \;
	 \frac{|2g_1 + g_2|}{\sqrt{5}}
	 \right).
	 $
	 Thus we have,
	 \begin{equation}\label{vector eqn}
	 f_1 + f_2 = \pm (g_1 + g_2), 
	 f_1 - f_2 = \pm (g_1 - g_2), 
	 2f_1 + f_2 = \pm (2g_1 + g_2).
	 \end{equation}
	 In this case there are total eight possibilities for checking the injectivity of the  equation (\ref{vector eqn}). However, six possibilities are not well defined.  So we consider the well-defined choies accordingly. 
	 
	 \vspace{0.2cm}
	 
	 {\bf Subcase 1.1:} 
	If  $
	 f_1 + f_2 =  (g_1 + g_2), ~
	 f_1 - f_2 = (g_1 - g_2) ~\text{and}~
	 2f_1 + f_2 = (2g_1 + g_2).$
	 	Then it is easy to verify that $f_1 = g_1$ and $f_2= g_2$.
	 	
	 	 \vspace{0.2cm}
	 	
	 	{\bf Subcase 1.2:} 
	 	 Furthermore, if $	f_1 + f_2 =  -(g_1 + g_2), ~
	 	f_1 - f_2 = -(g_1 - g_2),~\text{and}~
	 	2f_1 + f_2 = -(2g_1 + g_2).$
	 	Thus we obtain $f_1 = -g_1$ and $f_2= -g_2.$ 
	 	
	 	Therefore $\hat{f} = \hat{g}$ and hence $\gamma_\sigma$ is injective.
	 	
	 	 \vspace{0.2cm}
	 
	 	{\bf Case 2:}  $\sigma$ = $\{1\}$ and suppose
	 	$$\left( (\gamma_\sigma(\hat{f})_i) \right)_i^{3} = \left( (\gamma_\sigma(\hat{g})_i) \right)_i^{3}.$$
	 %	\[
	 %	\implies \left( 1\|P_{V_1} f\|,\; 1\|P_{W_2} f\|,\; 1\|P_{W_3} f\| \right)
	 %	= 
	 %	\left( 1\|P_{V_1} g\|,\; 1\|P_{W_2} g\|,\; 1\|P_{W_3} g\| \right)
	 %	\]
	 	Therefore, $\left( 
	 	{|f_1 |}{}, \;
	 	\frac{|f_1 - f_2|}{\sqrt{2}}, \;
	 	\frac{|2f_1 + f_2|}{\sqrt{5}}
	 	\right)
	 	=
	 	\left( 
	 	{|g_1|}{}, \;
	 	\frac{|g_1 - g_2|}{\sqrt{2}}, \;
	 	\frac{|2g_1 + g_2|}{\sqrt{5}}
	 	\right)
	 	$.
	 	Hence we have ,
	 	$
	 	f_1  = \pm g_1 ,  ~
	 	f_1 - f_2 = \pm (g_1 - g_2) ~\text{and} ~ 
	 	2f_1 + f_2 = \pm (2g_1 + g_2).
	 	$
	 %	There are total 8-possibilities for checking the injectivity of the  associated vectors, However, 6-possibilities are not well defined. Because they do not satisfy the above vector equations. So we consider the well-defined choies accordingly. 
	 	
	 %	{\bf Subcase 2.1:} 
	 %	\[
	 %	f_1  =  g_1, \quad 
	 %	f_1 - f_2 = (g_1 - g_2) \quad \text{and} \quad 
	 %	2f_1 + f_2 = (2g_1 + g_2).
	 %	\]
	 	%Clearly $f_1 = g_1$ and $f_2= g_2$
	 	
	 %	{\bf Subcase 2.2:} 
	 %	\[
	 %	f_1  =  -g_1, \quad 
	 %	f_1 - f_2 = -(g_1 - g_2) \quad \text{and} \quad 
	 %	2f_1 + f_2 = -(2g_1 + g_2).
	 %	\]
	 %	Clearly $f_1 = -g_1$ and $f_2= -g_2.$ 
	 	
	 Using a similar argument as in the previous case, we establish that $\hat{f} = \hat{g}$ and consequetly $\gamma_\sigma$ is injective. Similarly, we can show that $\gamma_\sigma$ is injective in all other cases.
	 \end{ex}
	 
	The following example demonstrates the Definition 3.1 does not ensure phase retrievability for all weaving fusion frames in Hilbert spaces.
	 
	\begin{ex}\label{ex4}
		
		Let us consider a separable Hilbert space $\mathcal{H} = \mathbb{C}^2$. Suppose
		$	 		
		\{(\mathcal V_i, 1)\}_{i=1}^3 = \left\{
		\left( \operatorname{span} \left( \begin{bmatrix} 1 \\ 0 \end{bmatrix} \right), 1 \right),
		\left( \operatorname{span} \left( \begin{bmatrix} 0 \\ 1 \end{bmatrix} \right), 1 \right),
		\left( \operatorname{span} \left( \begin{bmatrix} 1 \\ 2 \end{bmatrix} \right), 1 \right)
		\right\}
		$
		
		and
		
		$
		\{(\mathcal W_i, 1)\}_{i=1}^3 =\left\{
		\left( \operatorname{span} \left( \begin{bmatrix} \frac{2}{\sqrt{2}} \\ \frac{1}{\sqrt{2}} \end{bmatrix} \right), 1 \right),
		\left( \operatorname{span} \left( \begin{bmatrix} \frac{1}{\sqrt{2}} \\ -\frac{1}{\sqrt{2}} \end{bmatrix} \right), 1 \right),
		\left( \operatorname{span} \left( \begin{bmatrix} 1 \\ 1 \end{bmatrix} \right), 1 \right)
		\right\}
		$	 						
	 
		%It is easy to see that for every subset $\sigma \subset \{1, 2, 3\},$ equation ~\eqref{eq1} is satisfied. Thus, 
		  are fusion frames for $\mathbb C^2$.
		  Then it is easy to verify that they are weaving fusion frames, however they are not phase retrievable weaving fusion frames.
		   Here $\sigma = \{1, 2\}$ with $\gamma_\sigma: \hat{\mathbb C^2} \rightarrow \mathbb C^3$ defined by  
		\[
		\left( (\gamma_\sigma(\hat{f})_i) \right)_i^{3} =
		\begin{cases}
			v_i \| P_{V_i} f \| & : i \in \sigma \\
			w_i \| P_{W_i} f \| & : i \in \sigma^c,
		\end{cases}
		\]
		where $\sigma \subset \{1, 2, 3\}$. Consider two distinct vectors 
		$x =  \begin{bmatrix} 1 \\ i \end{bmatrix}$ and $y =  \begin{bmatrix} i \\ 1 \end{bmatrix}$ in the Hilbert space $\mathcal{H} = \mathbb{C}^2.$ Since y is not a unimodular multiple of x, the two are not equivalent in this sense. Yet, by a simple calculation, we see that $\left( (\gamma_\sigma(\hat{x})_i) \right)_i^{3} = \left( (\gamma_\sigma(\hat{y})_i) \right)_i^{3}$ coordinate-wise.
		Thus $\gamma_\sigma$ is not injective.

	\end{ex}
	 
%$((\gamma_{\sigma}(\hat f))_{i})_{i= 1}^m=\begin{cases}
%&	v_i\|P_{\mathcal V_i}f\| \text{ if } ii...\\
%	&w_i\|P_{\mathcal W_i}f\| .
%\end{cases}$

The following result provides a characterization of phase retrievable weaving fusion frames by analyzing the kernel of the phase lift operator. This is further linked to the structure of a symmetric operator whose range has dimension at most two.

\begin{thm}\label{thm1}
Let $\mathcal V_v=\{(\mathcal V_i, v_i)\}_{i=1}^m$ and  $\mathcal W_w=\{(\mathcal W_i, w_i)\}_{i=1}^m$ be two weaving fusion frames for $\mathcal H^n$. Then the following are equivalent:
\begin{enumerate}
\item $\{(\mathcal V_i, v_i)\}_{i=1}^m$ and $\{(\mathcal W_i, w_i)\}_{i=1}^m$ are  phase retrievable weaving fusion frames for $\mathcal H^n$.

\item $\text{Ker} \mathcal F_\sigma \cap S^{1,1} (\mathcal H^n)=\{0\}$, where $\sigma \subset \{1, 2, \cdots , m\}$.
\end{enumerate}
\end{thm}

\begin{proof}
\noindent (\underline{1 $\implies$ 2}) Let $\{(\mathcal V_i, v_i)\}_{i=1}^m$ and $\{(\mathcal W_i, w_i)\}_{i=1}^m$ be  phase retrievable weaving fusion frames for $\mathcal H^n$. Then for every subset $\sigma \subset \{1, 2, \cdots , m\}$ the operator

 $\gamma_\sigma: \mathcal H^n \rightarrow \mathbb R^m$ defined by  
\[
\left( (\gamma_\sigma(\hat{f})_i) \right)_i^{m} =
\begin{cases}
	v_i \| P_{V_i} f \| & : i \in \sigma \\
	w_i \| P_{W_i} f \| & : i \in \sigma^c
\end{cases}
\]
is one-one, where $\sigma\subset \{1, 2, \cdots, m\}$. 
Then applying Theorem \ref{Balan}, for every $T \in \text{Ker} \mathcal F_\sigma \cap S^{1,1} (\mathcal H^n)$ there exist $f, g \in \mathcal H^n$ so that $T=f\otimes f - g\otimes g$ and $\mathcal F_\sigma(T)=0$. Therefore, for every subset $\sigma \subset \{1, 2, \cdots , m\}.$  we have,
\begin{eqnarray*}
	\left( (\mathcal{F}_\sigma)_i \right)_{i=1}^m & = &
	\left\{
	\begin{array}{ll}
		v_i^2 \,\operatorname{tr}(P_{V_i} T) & : i \in \sigma \\
		w_i^2 \,\operatorname{tr}(P_{W_i} T) & : i \in \sigma^c
		
	\end{array}
	\right.
	\\&=& 
		\left\{
	\begin{array}{ll}
		v_i^2 \,\operatorname{tr}(P_{V_i} (f\otimes f - g\otimes g)) & : i \in \sigma \\
		w_i^2 \,\operatorname{tr}(P_{W_i} (f\otimes f - g\otimes g)) & : i \in \sigma^c
	\end{array}
	\right.
\\&=&
0.
\end{eqnarray*}

%	\begin{eqnarray*}
%((\mathcal F_\sigma)_{i})_{i= 1}^m &=& v_i ^2 tr(P_{\mathcal V_i}T)+w_j ^2 tr(P_{\mathcal W_j}T)
%\\&=& v_i ^2 tr(P_{\mathcal V_i}(f\otimes f - g\otimes g))+w_j ^2 tr(P_{\mathcal W_j}(f\otimes f - g\otimes g))
%\\&=& 0.
%\end{eqnarray*}
Hence we have, $$v_i^2 \,\operatorname{tr}(P_{V_i} (f\otimes f))=v_i^2 \,\operatorname{tr}(P_{V_i} (g\otimes g))$$ and 
$$w_i^2 \,\operatorname{tr}(P_{W_i} (f\otimes f))=w_i^2 \,\operatorname{tr}(P_{W_i} (g\otimes g)) $$

Thus we obtain $v_i \| P_{V_i} f \|  = 	v_i \| P_{V_i} g \|$ and  $ w_i \| P_{W_i} f \|  = 	w_i \| P_{W_i} g \|.$

Since $\gamma_\sigma$ is one-one, there exists $h\in\mathbb T$ such that $g=hf$ i.e. $\hat{f} = \hat{g}$ . It follows that $$f\otimes f  = hg\otimes gh^* = g\otimes g .$$Therefore, $T=f\otimes f - g\otimes g=0$. Thus $\text{Ker} \mathcal F_\sigma \cap S^{1,1} (\mathcal H^n)=\{0\}$.
\\

\noindent (\underline{2 $\implies$ 1}) Conversely, if $\gamma_\sigma(\hat f)=\gamma_\sigma(\hat g)$, then for every subset $\sigma \subset \{1, 2, \cdots , m\}.$ 
 we have,
  \begin{eqnarray*}
  \left( (\gamma_\sigma(\hat{f})_i) \right)_i^{m} =
  \begin{cases}
  	v_i \| P_{V_i} f \| \\
  	w_i \| P_{W_i} f \| 
  \end{cases} = & \begin{cases}
  v_i \| P_{V_i} g \| \\
  w_i \| P_{W_i} g \| 
  \end{cases} =  \left( (\gamma_\sigma(\hat{g})_i) \right)_i^{m}.
  \end{eqnarray*}

Therefore, $\left((\mathcal F_\sigma(f\otimes f - g\otimes g))_i \right)
_i^m=0$. Thus $f\otimes f - g\otimes g \in \text{Ker} \mathcal F_\sigma \cap S^{1,1} (\mathcal H^n)$. Since $\text{Ker} \mathcal F \cap S^{1,1} (\mathcal H^n)=\{0\}$, we have $f\otimes f = g\otimes g $.  Hence we say that there exists $h\in \mathbb T$ so that $g=hf$ i.e. $\hat f = \hat g$. Therefore, $\gamma_\sigma$ is one-one. Consequently, $\{(\mathcal V_i, v_i)\}_{i=1}^m$ and $\{(\mathcal W_i, w_i)\}_{i=1}^m$ are  phase retrievable weaving fusion frames for $\mathcal H^n$.
\end{proof} 

The following theorem provides validation for the above characterization result.

\begin{ex}
	Consider a separable Hilbert space $\mathcal{H} = \mathbb{R}^3$. Suppose
	\[
	\{(\mathcal V_i, 1)\}_{i=1}^5 = \left\{
	\begin{aligned}
		&\left( \operatorname{span}\left( \begin{bmatrix} 1 \\ 0 \\ 0 \end{bmatrix} \right), 1 \right),
		\left( \operatorname{span}\left( \begin{bmatrix} 0 \\ 1 \\ 0 \end{bmatrix} \right), 1 \right),
		\left( \operatorname{span}\left( \begin{bmatrix} 0 \\ 0 \\ 1 \end{bmatrix} \right), 1 \right), \\[6pt]
		&\left( \operatorname{span}\left( \begin{bmatrix} 1 \\ -1 \\ 1 \end{bmatrix} \right), 1 \right),
		\left( \operatorname{span}\left( \begin{bmatrix} 2/3 \\ 2 \\ -3 \end{bmatrix} \right), 1 \right)
	\end{aligned}
	\right\}
	\]
	and 
	\[
	\{(\mathcal V_i, 1)\}_{i=1}^5 = \left\{
	\begin{aligned}
		&\left( \operatorname{span}\left( \begin{bmatrix} 5 \\ -3/2 \\ 0 \end{bmatrix} \right), 1 \right),
		\left( \operatorname{span}\left( \begin{bmatrix} -1 \\ 3 \\ 3/2 \end{bmatrix} \right), 1 \right),
		\left( \operatorname{span}\left( \begin{bmatrix} -1 \\ -3 \\ 2/3 \end{bmatrix} \right), 1 \right), \\[6pt]
		&\left( \operatorname{span}\left( \begin{bmatrix} 1 \\ 1 \\ 1 \end{bmatrix} \right), 1 \right),
		\left( \operatorname{span}\left( \begin{bmatrix} 1 \\ 2 \\ 3 \end{bmatrix} \right), 1 \right)
	\end{aligned}
	\right\}
	\]
	are weaving fusion frames for $\mathcal{H} = \mathbb{R}^3$.
	
	\noindent (\underline{1 $\implies$ 2}):
	
It is easy to verify that $\{(\mathcal V_i, v_i)\}_{i=1}^5$ and $\{(\mathcal W_i, w_i)\}_{i=1}^5$ are phase retrievable weaving fusion frames. Then by definition of weaving phase lift operator, for every $\sigma \subset \{1, 2, 3, 4, 5\}$, $\mathcal F_\sigma: Sym(\mathbb R^3) \rightarrow \mathbb R^5$ is defined as, 

	$$\left( (\mathcal{F}_\sigma T)_i \right)_{i=1}^5 =
	\begin{cases}
		v_i^2 \, \operatorname{tr}(P_{V_i} T) & : i \in \sigma \\
		w_i^2 \, \operatorname{tr}(P_{W_i} T) & : i \in \sigma^c
	\end{cases}, ~\text{where} ~T =
	\begin{bmatrix}
		a & b & c \\
		b & d & e \\
		c & e & f
	\end{bmatrix}.$$

	{\bf Case 1:} $\sigma = \{1, 2, 3\}$
	\begin{eqnarray*}
		\left( (\mathcal{F}_\sigma T)_i \right)_{i=1}^m &=&
	\left(
	 v_{1}^{2} tr(P_{V_{1}}(T)),\,
	v_{2}^{2} tr(P_{V_{2}}(T)),\,
	v_{3}^{2} tr(P_{V_{3}}(T)),\,
	w_{4}^{2} tr(P_{W_{4}}(T)),\,
	w_{5}^{2} tr(P_{W_{5}}(T)
		\right).
		\\&=& 
		\left(
		a, d, f, \tfrac{1}{3}(a+2b+2c+d+2e+f), \tfrac{1}{14}(a+4b+6c+4d+12e+9f) 
	\right).
	\end{eqnarray*}
	Therefore,
	
$\text{Ker} \mathcal F_\sigma(T)
	= \Big\{\, T =
	\begin{bmatrix}
		0 & b & c \\
		b & 0 & e \\
		c & e & 0
	\end{bmatrix}
	\;\Big|\; b=3t,\; c=-4t,\; e=t \Big\}$. Thus  $\text{Ker} (\mathcal F_\sigma(T))$ is either a collection of invertiable matrices or zero matrix. Hence we obtain $\text{Ker} \mathcal F_\sigma(T) \cap S^{1,1} (\mathbb R^3)=\{0\}.$\\
	
	{\bf Case 2:} $\sigma = \{1, 2, 3, 4\}$
	\begin{eqnarray*}
		\left( (\mathcal{F}_\sigma T)_i \right)_{i=1}^m & = &
		\left(
		v_{1}^{2} tr(P_{V_{1}}(T)),\,
		v_{2}^{2} tr(P_{V_{2}}(T)),\,
		v_{3}^{2} tr(P_{V_{3}}(T)),\,
		v_{4}^{2} tr(P_{V_{4}}(T)),\,
		w_{5}^{2} tr(P_{W_{5}}(T)) 
		\right).
	\end{eqnarray*}
	Applying foregoing argument as in the previous case, we obtain $\text{Ker} \mathcal F_\sigma(T) \cap S^{1,1} (\mathbb R^3)=\{0\}.$
	 For all the remaining cases, the proof can be carried out in a similar manner.\\
	 
	 \noindent (\underline{2 $\implies$ 1}): It is evident that $\{(\mathcal V_i, v_i)\}_{i=1}^5$ and $\{(\mathcal W_i, w_i)\}_{i=1}^5$ are weaving fusion frames and  $\text{Ker} (\mathcal F_\sigma(T) )\cap S^{1,1}(\mathbb R^3)=\{0\}.$
	Therefore, it is clear to see that for every $\sigma \subset \{1, 2, 3, 4, 5\},$  $\left( (\gamma_\sigma(\hat{f})_i) \right)_i^{5}$ is injective and hence the frames under consideration are phase retrievable weaving fusion frames.	
\end{ex}

The following result presents a necessary and sufficient condition for two fusion frames to form phase retrievable weaving fusion frames.

\begin{thm}
Let $f, g \in \mathcal H^n$ and $[f,g]=\frac{1}{2}(f\otimes g + g\otimes f)$. Then for any two weaving fusion frames $\{(\mathcal V_i, v_i)\}_{i=1}^m$ and $\{(\mathcal W_i, w_i)\}_{i=1}^m$ in $\mathcal H^n$, the following are equivalent:
\begin{enumerate}
\item $\{(\mathcal V_i, v_i)\}_{i=1}^m$ and $\{(\mathcal W_i, w_i)\}_{i=1}^m$ are phase retrievable weaving fusion frames for $\mathcal H^n$.

\item For every $f, g \in \mathcal H^n$ and for every $\sigma \subset \{1, 2, \cdots , m\}$, there exists a constant $\alpha>0$ so that
\begin{eqnarray*}
&&\frac{1}{4} \left \{\sum\limits_{i\in \sigma} \|v_i ^2 \langle P_{\mathcal V_i}g,f \rangle + v_i ^2 \langle P_{\mathcal V_i}f,g \rangle\|^2 + \sum\limits_{i\in \sigma^c} \|w_i ^2 \langle P_{\mathcal W_i}g,f \rangle + w_i ^2 \langle P_{\mathcal W_i}f,g \rangle\|^2 \right\}
\\&=& \sum\limits_{i\in \sigma} \|v_i ^2 tr(P_{\mathcal V_i}[f,g])\|^2 + \sum\limits_{i\in \sigma^c} \|w_i ^2 tr(P_{\mathcal W_i}[f,g])\|^2
\\&\geq& \alpha \|[f,g]\|_{1}.
\end{eqnarray*}
\end{enumerate}
\end{thm}

\begin{proof}
\noindent (\underline{1 $\implies$ 2}) For every $f,g\in \mathcal H^n$ and for every $\sigma \subset \{1, 2, \cdots , m\}$ we have,
\begin{eqnarray*}
&&\frac{1}{2} v_i ^2 \langle P_{\mathcal V_i}g,f \rangle + \frac{1}{2} v_i ^2 \langle P_{\mathcal V_i}f,g \rangle 
%+ \frac{1}{2} w_i ^2 \langle P_{\mathcal W_i}g,f \rangle + \frac{1}{2} w_i ^2 \langle P_{\mathcal W_i}f,g \rangle
\\ &=& \frac{1}{2} v_i ^2 tr(P_{\mathcal V_i}g \otimes f)+\frac{1}{2} v_i ^2 tr(P_{\mathcal V_i}f \otimes g) 
% + {1}{2} w_i ^2 tr(P_{\mathcal W_i}g \otimes f) + \frac{1}{2} w_i ^2 tr(P_{\mathcal W_i}f \otimes g)
\\&=& (\mathcal F_\sigma([f,g]))_i.
\end{eqnarray*}
and
\begin{eqnarray*}
	&&\frac{1}{2} v_i ^2 \langle P_{\mathcal W_i}g,f \rangle + \frac{1}{2} v_i ^2 \langle P_{\mathcal W_i}f,g \rangle 
	%+ \frac{1}{2} w_i ^2 \langle P_{\mathcal W_i}g,f \rangle + \frac{1}{2} w_i ^2 \langle P_{\mathcal W_i}f,g \rangle
	\\ &=& \frac{1}{2} v_i ^2 tr(P_{\mathcal W_i}g \otimes f)+\frac{1}{2} v_i ^2 tr(P_{\mathcal W_i}f \otimes g) 
	% + {1}{2} w_i ^2 tr(P_{\mathcal W_i}g \otimes f) + \frac{1}{2} w_i ^2 tr(P_{\mathcal W_i}f \otimes g)
	\\&=& (\mathcal F_\sigma([f,g]))_i.
\end{eqnarray*}
Applying Theorem \ref{thm1}, for every $f,g\in \mathcal H^n$ with $[f,g]\neq 0$ we have $\mathcal F_\sigma([f,g])\neq 0$. Therefore,  for every $\sigma \subset \{1, 2, \cdots , m\}$ we have,
\begin{eqnarray*}
\|(\mathcal F_\sigma([f,g]))\|^2 &=&\frac{1}{4} \sum\limits_{i\in \sigma} \|v_i ^2 \langle P_{\mathcal V_i}g,f \rangle + v_i ^2 \langle P_{\mathcal V_i}f,g \rangle\|^2 
\\&+& \frac{1}{4} \sum\limits_{i\in \sigma^c} \|w_i ^2 \langle P_{\mathcal W_i}g,f \rangle + w_i ^2 \langle P_{\mathcal W_i}f,g \rangle\|^2
\\&>&0.
\end{eqnarray*}
Thus we obtain,
\begin{eqnarray*}
&&\frac{1}{4} \sum\limits_{i\in \sigma} \|v_i ^2 \langle P_{\mathcal V_i}g,f \rangle + v_i ^2 \langle P_{\mathcal V_i}f,g \rangle\|^2 
+ \frac{1}{4} \sum\limits_{i\in \sigma^c} \|w_i ^2 \langle P_{\mathcal W_i}g,f \rangle + w_i ^2 \langle P_{\mathcal W_i}f,g \rangle\|^2
\\&=& \sum\limits_{i\in \sigma} \left\|\frac{1}{2} v_i ^2 tr(P_{\mathcal V_i}g \otimes f)+\frac{1}{2} v_i ^2 tr(P_{\mathcal V_i}f \otimes g)\right\|^2 
\\&+& \sum\limits_{i\in \sigma^c} \left\| \frac{1}{2} w_i ^2 tr(P_{\mathcal W_i}g \otimes f) + \frac{1}{2} w_i ^2 tr(P_{\mathcal W_i}f \otimes g)\right\|^2
\\&=& \sum\limits_{i\in \sigma} \|v_i ^2 tr(P_{\mathcal V_i}[f,g])\|^2 + \sum\limits_{i\in \sigma^c} \|w_i ^2 tr(P_{\mathcal W_i}[f,g])\|^2
\\&>&0.
\end{eqnarray*}
For every $\sigma \subset \{1, 2, \cdots , m\}$ let us assume
$$\alpha = \min\limits_{T\in S^{1,1}(\mathcal H^n), \|T\|_{1}} \left(\sum\limits_{i\in \sigma} \|v_i ^2 tr(P_{\mathcal V_i}T\|^2 + \sum\limits_{i\in \sigma^c} \|w_i ^2 tr(P_{\mathcal W_i}T)\|^2\right).$$
Since $\frac{f\otimes g} {\|[f,g]\|_{1}} \in  S^{1,1}(\mathcal H^n)$ and $\left\| \frac{[f,g]}{\|[f,g]\|_{1}} \right\|_{1}=1$, then we have,
\begin{eqnarray*}
&& \sum\limits_{i\in \sigma} \frac{1}{\|[f,g]\|_{1} ^2} \|v_i ^2 tr(P_{\mathcal V_i}[f,g])\|^2 + \sum\limits_{i\in \sigma^c} \frac{1}{\|[f,g]\|_{1} ^2} \|w_j ^2 tr(P_{\mathcal W_i}[f,g])\|^2
 \\&=& \sum\limits_{i\in \sigma}  \left\|v_i ^2 tr \left(P_{\mathcal V_i}\frac{[f,g]}{\|[f,g]\|_{1}}\right)\right\|^2 + \sum\limits_{i\in \sigma^c} \left\|w_i ^2 tr \left(P_{\mathcal W_i}\frac{[f,g]}{\|[f,g]\|_{1}}\right)\right\|^2
 \\&\geq & \alpha.
\end{eqnarray*}
Thus for every $f, g \in \mathcal H^n$ and for every $\sigma \subset \{1, 2, \cdots , m\}$ we obtain the desired result.
\\

\noindent (\underline{2 $\implies$ 1}) This implication will be established using a similar approach as employed in the converse direction.
\end{proof}

The next result demonstrates that the image of phase retrievable weaving fusion frames under a unitary operator remains a phase retrievable weaving fusion frame.

\begin{prop}
Let $\{(\mathcal V_i, v_i)\}_{i=1}^m$ and $\{(\mathcal W_i, w_i)\}_{i=1}^m$ be  phase retrievable weaving fusion frames for $\mathcal H^n$ with bounds $\alpha, \beta$. Suppose $Q\in\mathcal L(\mathcal H^n)$ is a unitary operator, then $\{(Q\mathcal V_i, v_i)\}_{i=1}^m$ and $\{(Q\mathcal W_i, w_i)\}_{i=1}^m$ also form  phase retrievable weaving fusion frames for $\mathcal H^n$.
\end{prop}

\begin{proof}
For every $1\leq i \leq m$, the projection operator $P_{Q\mathcal V_i}$ onto $Q\mathcal V_i$ is defined as $P_{Q\mathcal V_i}=Q P_{\mathcal V_i} Q^*$ (see \cite{Ga07}), analogously $P_{Q\mathcal W_i}=Q P_{\mathcal W_i} Q^*$.

Let $f, g \in \mathcal H^n$ so that for every $1\leq i \leq m$, $v_i \|P_{Q\mathcal V_i} f\|=v_i \|P_{Q\mathcal V_i} g\|$. Then we have, 
$v_i \|P_{Q\mathcal V_i} f\|=v_i \|Q P_{\mathcal V_i} Q^* f\|=v_i \| P_{\mathcal V_i} Q^* f\|=v_i \|Q P_{\mathcal V_i} Q^* g\|
=v_i \| P_{\mathcal V_i} Q^* g\|$ 

and $w_i \|P_{Q\mathcal W_i} f\|=w_i \|Q P_{\mathcal W_i} Q^* f\|=w_i \| P_{\mathcal W_i} Q^* f\|=w_i \|Q P_{\mathcal W_i} Q^* g\|=w_i \| P_{\mathcal W_i} Q^* g\|$.

Since $\{(\mathcal V_i, v_i)\}_{i=1}^m$ and $\{(\mathcal W_i, w_i)\}_{i=1}^m$ are  phase retrievable weaving fusion frames there exists $h \in \mathbb T$ such that $Q^*f=hQ^*g$ i.e. $Q^*(f-hg)=0.$ Furthermore, since $Q$ is unitary, we have $f=hg$. It is sufficient to prove that $\{(Q\mathcal V_i, v_i)\}_{i=1}^m$ and $\{(Q\mathcal W_i, w_i)\}_{i=1}^m$ are weaving fusion frames for $\mathcal H^n$.

Since $\{(\mathcal V_i, v_i)\}_{i=1}^m$ and $\{(\mathcal W_i, w_i)\}_{i=1}^m$ are weaving fusion frames, then for every $f\in\mathcal H^n$ and every $\sigma \subset \{1, 2, \cdots , m\}$ we obtain,
\begin{eqnarray*}
\alpha \|f\|^2=\alpha \|Q^*f\|^2 &\leq & \sum\limits_{i\in\sigma} v_i ^2 \| P_{\mathcal V_i} Q^* f\|^2 + \sum\limits_{j\in\sigma^c} w_j ^2 \| P_{\mathcal W_j} Q^* f\|^2
\\ &=& \sum\limits_{i\in\sigma} v_i ^2 \| Q P_{\mathcal V_i} Q^* f\|^2 + \sum\limits_{j\in\sigma^c} w_j ^2 \|Q P_{\mathcal W_j} Q^* f\|^2
\\ &=& \sum\limits_{i\in\sigma} v_i ^2 \| P_{Q\mathcal V_i} f\|^2 + \sum\limits_{j\in\sigma^c} w_j ^2 \| P_{Q\mathcal W_j} f\|^2
\\&\leq & \beta \|Q^*f\|^2=\beta \|f\|^2.
\end{eqnarray*}
Thus  $\{(Q\mathcal V_i, v_i)\}_{i=1}^m$ and $\{(Q\mathcal W_i, w_i)\}_{i=1}^m$ are weaving fusion frames for $\mathcal H^n$. This completes the proof.
\end{proof}

One can verify the phase retrievability of weaving fusion frames by implementing the given algorithm. The following code provides a practical way to check this property.

\section*{Python Code for Complement Property of Phase Retrievable Weaving Fusion Frames}
\begin{lstlisting}[language=Python, caption=Checking Complement Property for Weaving Frames]
	import numpy as np
	import itertools
	
	def has_complement_property(frame, tol=1e-10, verbose=True):
	"""
	Check complement property (CP) for a finite frame in R^d.
	"""
	m = len(frame)
	d = frame[0].shape[0]
	F = np.column_stack(frame)
	
	for r in range(1, m):  # exclude empty and full set
	for subset in itertools.combinations(range(m), r):
	sub_matrix = F[:, subset]
	rank_subset = np.linalg.matrix_rank(sub_matrix, tol)
	complement = [i for i in range(m) if i not in subset]
	comp_matrix = F[:, complement]
	rank_complement = np.linalg.matrix_rank(comp_matrix, tol)
	
	if verbose:
	print(f"Subset {subset} -> rank={rank_subset}, "
	f"Complement {complement} -> rank={rank_complement}")
	
	
	if rank_subset < d and rank_complement < d:
	if verbose:
	print(" Complement property fails here!\n")
	return False
	if verbose:
	print("Complement property satisfied for all subsets.\n")
	return True
	
	# Example in R^3
	v1 = np.array([1,0,0])
	v2 = np.array([0,1,0])
	v3 = np.array([0,0,1])
	v4 = np.array([1,-1,1])
	v5 = np.array([2/3,2,-3])
	
	w1 = np.array([5,-3/2,0])
	w2 = np.array([-1,3,3/2])
	w3 = np.array([-1,-3,2/3])
	w4 = np.array([1,1,1])
	w5 = np.array([1,2,3])
	
	V = [v1, v2, v3, v4, v5]
	W = [w1, w2, w3, w4, w5]
	
	combinations = []
	for choice in itertools.product([0,1], repeat=5):
	frame = [V[i] if choice[i]==0 else W[i] for i in range(5)]
	combinations.append(frame)
	
	for i, frame in enumerate(combinations, 1):
	print("="*40)
	print(f"Checking Combination {i}:")
	has_complement_property(frame, verbose=True)
\end{lstlisting}

 \section{Application to Probabilistic Erasure}
 
In this section, we explore the applications of weaving frames within the framework of probabilistic erasure. We focus on their robustness and stability in recovering data when random loss occurs. Weaving fusion frames offer redundancy, which proves essential in such uncertain environments. This discussion provides foundational insight into their role in signal reconstruction under erasure models. It also prepares the ground for more detailed analysis on their effectiveness in probabilistic data recovery scenarios.
 
Let $\{(\mathcal V_n,1)\}_{n=1}^\infty$ and $\{(\mathcal W_n,1)\}_{n=1}^\infty$ be weaving fusion frames for $\mathcal H$ with the corresponding frame operators $S_{\mathcal V}, S_{\mathcal W}$, then for every $f\in\mathcal H$ and $\sigma \subset \mathbb N$, the associated error operator is given by,
\begin{equation*}
E=\sum\limits_{n \in \sigma} \delta_n P_{\mathcal V_n} f \otimes S_{\mathcal V}^{-1} P_{\mathcal V_n} f + \sum\limits_{n \in \sigma^c} \delta_n P_{\mathcal W_n} f \otimes S_{\mathcal W}^{-1} P_{\mathcal W_n} f,
\end{equation*}
where $\delta_n$ is the standard dirac delta function. 
 
This error operator enables rapid data representation while considerably reducing computational cost. Its performance heavily depends on the selection of weaving frames used in the encoding and decoding phases. Optimizing these frame choices can significantly enhance the technique's overall efficiency and robustness.

In the following result, we demonstrate how weaving frames can effectively mitigate data loss in erasure scenarios.

\begin{thm}
Let $\{(\mathcal V_n,1)\}_{n=1}^m$ and $\{(\mathcal W_n,1)\}_{n=1}^m$ be uniform tight weaving frames for  $\mathbb R^n$ and for every $f\in \mathbb R^n$ $\|P_{\mathcal V_n} f \| = \sqrt n = \|P_{\mathcal W_n} f \|$. Suppose $\epsilon>0$ is very small number for which $\epsilon^2 \geq \frac{n}{m} \text{log} (n)$. Then for every $f \in \mathbb R^n$, there is an absolute constant $M>0$ for which we have, $$E\|\hat f-f\| \leq M \epsilon \|f\|.$$
\end{thm}

\begin{proof}
For every $f \in \mathbb R^n$ and every $\sigma \subset \{1, 2, \cdots, m\}$ we have,
\begin{equation*}
\tilde{E} f = \hat f-f = \frac{1}{m} \left [\sum\limits_{n \in \sigma} \epsilon_n P_{\mathcal V_n} f  + \sum\limits_{n \in \sigma^c} \epsilon_n P_{\mathcal W_n} f  \right],
\end{equation*}
where 
\begin{align*}
\hat f & = \frac{2}{m}  \left [\sum\limits_{n \in \sigma} P_{\mathcal V_n} f  + \sum\limits_{n \in \sigma^c} P_{\mathcal W_n} f  \right]\\ 
& =\frac{2}{m}  \left [\sum\limits_{n \in \sigma} \tilde{\delta_n}^{p_n} P_{\mathcal V_n} f  + \sum\limits_{n \in \sigma^c} \tilde{\delta_n}^{p_n} P_{\mathcal W_n} f  \right],
\end{align*}
 with
$$
\tilde{E}=\sum\limits_{n \in \sigma} \tilde{\delta_n}^{p_n} P_{\mathcal V_n} f  \otimes P_{\mathcal V_n} f  + \sum\limits_{n \in \sigma^c} \tilde{\delta_n}^{p_n} P_{\mathcal V_n} f  \otimes P_{\mathcal V_n} f  
~\mbox{ and }~
\epsilon_n = 2\tilde{\delta_n}^{p_n} - 1
$$ is an independent Bernoulli variable that takes values $1$ and $-1$, each with probability of $\frac{1}{2}$ (see \cite{Pa99}).

Thus it is sufficient to prove that, for every $f\in\mathbb R^n$ and for every $\sigma \subset \{1, 2, \cdots, m\}$,
\begin{equation*}
E\left \|  \frac{1}{m} \left [\sum\limits_{n \in \sigma} \epsilon_n P_{\mathcal V_n} f  \otimes P_{\mathcal V_n} f  + \sum\limits_{n \in \sigma^c} \epsilon_n  P_{\mathcal W_n} f  \otimes P_{\mathcal W_n} f  \right] \right \| \leq M \epsilon.
\end{equation*}
Applying [Lemma, \cite{Ru99}] for every $f\in\mathbb R^n$ and for every $\sigma \subset \{1, 2, \cdots, m\}$, there is a positive constant $M$ we have 

\begin{eqnarray}\label{app1}
&& E \left \| \left [\sum\limits_{n \in \sigma} \epsilon_n P_{\mathcal V_n} f  \otimes P_{\mathcal V_n} f  + \sum\limits_{n \in \sigma^c} \epsilon_n   P_{\mathcal W_n} f  \otimes P_{\mathcal W_n} f  \right] \right \|
\\\nonumber& \leq &  M \sqrt{n \text{log}(n)} 
 \left \| \left [\sum\limits_{n \in \sigma} P_{\mathcal V_n} f  \otimes P_{\mathcal V_n} f  + \sum\limits_{n \in \sigma^c}   P_{\mathcal W_n} f  \otimes P_{\mathcal W_n} f  \right] \right \|^{\frac{1}{2}}.
\end{eqnarray}

%\begin{equation}\label{app1}
%E\left \| \left [\sum\limits_{n \in \sigma} \epsilon_n z_n \otimes z_n + \sum\limits_{n \in \sigma^c} \epsilon_n  z'_n \otimes z'_n \right] \right \|_{J}
%\leq M \sqrt{n \text{log}(n)} \left \| \left [\sum\limits_{n \in \sigma} z_n \otimes z_n + \sum\limits_{n \in \sigma^c}  z'_n \otimes z'_n \right] \right \|_{J}^{\frac{1}{2}}.
%\end{equation}
Furthermore, since $\{(\mathcal V_n,1)\}_{n=1}^m$ and $\{(\mathcal W_n,1)\}_{n=1}^m$  are uniform tight weaving frames for   $\mathbb R^n$ and for every $f\in \mathbb R^n$, $\|P_{\mathcal V_n} f\| = \sqrt n = \|P_{\mathcal W_n} f\|$, then for every $\sigma \subset \{1, 2, \cdots, m\}$ we have,
\begin{equation*}
\mathcal I_{\mathbb R^n} =  \frac{1}{m} \left [\sum\limits_{n \in \sigma}  P_{\mathcal V_n} f  \otimes P_{\mathcal V_n} f + \sum\limits_{n \in \sigma^c}  P_{\mathcal W_n} f  \otimes P_{\mathcal W_n} f \right], 
\end{equation*}
where $\mathcal I_{\mathbb R^n}$ is the identity operator on  $\mathbb R^n$.

Hence we obtain,
\begin{equation}\label{app2}
\left \| \left [\sum\limits_{n \in \sigma} P_{\mathcal V_n} f  \otimes P_{\mathcal V_n} f + \sum\limits_{n \in \sigma^c}  P_{\mathcal W_n} f  \otimes P_{\mathcal W_n} f \right] \right \| =m.
\end{equation}
Thus using the equations (\ref{app1}) and (\ref{app2}) we obtain,
\begin{equation*}
E\left \|  \frac{1}{m} \left [\sum\limits_{n \in \sigma} \epsilon_n P_{\mathcal V_n} f  \otimes P_{\mathcal V_n} f + \sum\limits_{n \in \sigma^c} \epsilon_n  P_{\mathcal W_n} f  \otimes P_{\mathcal W_n} f \right] \right \| \leq M \sqrt{\frac{n}{m}} ~\sqrt {\text{log}(n)}.
\end{equation*}
Therefore, we achieved our result due to the fact $\epsilon^2 \geq \frac{n}{m} \text{log} (n)$.
\end{proof}

%\section*{Acknowledgment}
%The authors acknowledge the financial support of the Ministry of Human Resource Development (M.H.R.D.), Government of India. Furthermore, they would like to express their sincere gratitude to Dr. Divya Singh and Dr. Saikat Mukherjee for their guidance to prepare this article.

%\section*{References}

%\bibliography{references-frame}

 \end{document}